\documentclass[10pt]{article}
\usepackage{amsmath,amsthm,amsfonts,amssymb,amscd,enumerate,pictex,graphicx}

\theoremstyle{plain}

\newtheorem*{I(n)}{$\I(n)$}
\newtheorem*{II(n)}{$\II(n)$}
\newtheorem*{III(2k+1)}{$\III(2k+1)$}
\newtheorem*{III'(2k+1)}{$\III'(2k+1)$}
\newtheorem*{V(n)}{$\V(n)$}
\newtheorem*{EFI(n)}{$\EFI(n)$}
\newtheorem*{TR(n)}{$\TR(n)$}
\newtheorem*{EFTR(n)}{$\EFTR(n)$}
\newtheorem*{EFII(n)}{$\EFII(n)$}
\newtheorem*{EFV(n)}{$\EFV(n)$}

\newtheorem{EConjecture}[subsection]{The Euler Characteristic Conjecture}
\newtheorem{SConjecture}[subsection]{Singer's Conjecture}
\newtheorem{subSCConjecture}[subsection]{Singer's Conjecture for Coxeter groups}

\newtheorem{Theorem}[subsubsection]{Theorem} 
\newtheorem{Thm}[subsection]{Theorem}
\newtheorem{Prop}[subsection]{Proposition}
\newtheorem{Proposition}[subsubsection]{Proposition}

\newtheorem{Cor}[subsubsection]{Corollary}
\newtheorem{Lemma}[subsection]{Lemma}
\newtheorem{Lem}[subsubsection]{Lemma}

\theoremstyle{definition}

\newtheorem{Rem}[subsubsection]{Remark}

\newtheorem{Ex}[subsubsection]{Example}

\DeclareMathOperator{\EFI}{\bf EFI}
\DeclareMathOperator{\I}{\bf I}
\DeclareMathOperator{\II}{\bf II}
\DeclareMathOperator{\V}{\bf V}
\DeclareMathOperator{\TR}{\bf TR}
\DeclareMathOperator{\EFTR}{\bf EFTR}
\DeclareMathOperator{\EFII}{\bf EFII}
\DeclareMathOperator{\EFV}{\bf EFV}
\newcommand{\Card}{\operatorname{Card}}
\newcommand{\wh}{\widehat}

\newcommand{\cs}{\mathcal{S}}

\newcommand{\cf}{\mathcal{F}}
\newcommand{\cH}{\mathcal{H}}

\def\l{\operatorname{\ell}}
\newcommand{\ltwo}{\l^2}

\newcommand{\gS}{\Sigma}

\newcommand{\gO}{\Omega}

\newcommand{\gs}{\sigma}

\newcommand{\BS}{\mathbb{S}}

\newcommand{\BZ}{\mathbb{Z}}

\newcommand{\BN}{\mathbb{N}}

\numberwithin{equation}{section}
\begin{document}

\title{Statements and Dilemmas Regarding the $\ell^2$-homology of Coxeter groups}
\author{Timothy A. Schroeder\\
Murray State University}
\date{September 14, 2012}

\maketitle

\begin{abstract}    
We generalize the methods used in \cite{schroedereven} to provide a program for proving Singer's Conjecture for Coxeter systems.  Specifically, we consider \emph{even} Coxeter systems with nerves that are flag triangulations of $\BS^{n-1}$, $n=2k$.  We prove that Conjecture \ref{conj:singerc} in dimensions $n-2$ and $n-1$, along with the vanishing of the $\ltwo$-homology of certain subspaces called ``two-letter'' ruins above dimension $k+1$, imply Conjecture \ref{conj:singerc} in dimension $n$.  This is but a program.  The author intends this paper to serve as a reference for those inquiring about Singer's Conjecture and about even Coxeter systems.  Users of this paper should focus attention on Sections \ref{ss:evenodd} and \ref{ss:evenconj}, along with Remark \ref{r:eftr}.

\end{abstract}


\section{Introduction}\label{s:intro} 
The following conjecture is attributed to Singer.
\begin{SConjecture}\label{conj:singer} If $M^{n}$ is a closed aspherical manifold, then the reduced $\ltwo$-homology of $\widetilde{M}^{n}$, $\cH_{\ast}(\widetilde{M}^n)$, vanishes for all $\ast\neq\frac{n}{2}$.
\end{SConjecture}
For details on $\ltwo$-homology theory, see \cite{davismoussong}, \cite{do2} and \cite{eckmann}.  Now, let $X$ be a geometric $G$-complex.  A key feature of the $\ell^2$-theory is that it is possible to attach to the Hilbert space $\cH_i(X)$ a nonnegative real number, called the $i^{\text{th}}$ $\ltwo$-Betti number.  A formula of Atiyah states that the alternating sum of these $\ltwo$-Betti numbers is the orbihedral Euler characteristic $\chi^{\text{orb}}(X/G)$, or in the case of a free action, the ordinary Euler characteristic $\chi(X/G)$.  Thus, Conjecture \ref{conj:singer} implies the following conjecture regarding Euler characteristic (attributed to H.Hopf):
\begin{EConjecture} If $M^{2k}$ is a closed, aspherical manifold of dimension $2k$, then its Euler characteristic, $\chi(M^{2k})$, satisfies
\[(-1)^k\chi(M^{2k})\geq 0.\]
\end{EConjecture}

Singer's conjecture holds for elementary reasons in dimensions $\leq 2$.  Indeed, top-dimensional cycles on manifolds are constant on each component, so a square-summable cycle on an infinite component is constant $0$.  As a result, Conjecture \ref{conj:singer} in dimension $\leq 2$ follows from Poincar\'e duality.  In \cite{LL}, Lott and L\"uck prove that it holds for those aspherical $3$-manifolds for which Thurston's Geometrization Conjecture is true.  (Hence, by Perelman, all aspherical $3$-manifolds.)  

Let $S$ be a finite set of generators.  A \emph{Coxeter matrix} on $S$ is a symmetric $S\times S$ matrix $M=(m_{st})$ with entries in $\BN\cup\{\infty\}$ such that each diagonal entry is $1$ and each off diagonal entry is $\geq 2$.  The matrix $M$ gives a presentation for an associated \emph{Coxeter} group $W$:
\begin{equation}\label{e:coxetergroup}
	W=\left\langle S\mid (st)^{m_{st}}=1, \text{ for each pair } (s,t) \text{ with } m_{st}\neq\infty\right\rangle.
\end{equation}
The pair $(W,S)$ is called a \emph{Coxeter system}.  Denote by $L$ the nerve of $(W,S)$.  In several papers (e.g., \cite{davisannals}, \cite{davisbook}, and \cite{davismoussong}), M. Davis describes a construction which associates to any Coxeter system $(W,S)$, a simplicial complex $\gS(W,S)$, or simply $\gS$ when the Coxeter system is clear, on which $W$ acts properly and cocompactly.  The two salient features of $\gS$ are that $(1)$ it is contractible and $(2)$ it permits a cellulation under which the link of each vertex is $L$.  It follows that if $L$ is a triangulation of $\BS^{n-1}$, $\gS$ is an $n$-manifold.  There is a special case of Singer's conjecture for such manifolds.  
 
\begin{subSCConjecture}\label{conj:singerc} Let $(W,S)$ be a Coxeter system such that its nerve, $L$, is a triangulation of $\BS^{n-1}$.  Then 
\[\cH_{i}(\gS(W,S))=0 \text{ for all } i\neq\frac{n}{2}.\]
\end{subSCConjecture}

In \cite{do2}, Davis and Okun prove that if Conjecture \ref{conj:singerc} for \emph{right-angled} Coxeter systems is true in some odd dimension $n$, then it is also true for right-angled systems in dimension $n+1$.  (A Coxeter system is right-angled if generators either commute or have no relation.)  They also show that Thurston's Geometrization Conjecture holds for these Davis $3$-manifolds arising from right-angled Coxeter systems.  Hence, the Lott and L\"uck result implies that Conjecture \ref{conj:singerc} for right-angled Coxeter systems is true for $n=3$ and, therefore, also for $n=4$.  (Davis and Okun also show that Andreev's theorem, \cite[Theorem 2]{andreev2}, implies Conjecture \ref{conj:singerc} in dimension $3$ for right-angled systems.)  In \cite{schroedergeom}, the author geometrizes arbitrary $3$-dimensional Davis manifolds and shows that Conjecture \ref{conj:singerc} in dimension $3$ follows. 

Right-angled Coxeter systems are specific examples of \emph{even} Coxeter systems.  We say a Coxeter system is even if for any two generators $s\neq t$, $m_{st}$ is either even or infinite.  In \cite{schroedereven}, the author proves the following extension of the Davis-Okun $4$-dimensional result:

\begin{Thm}\label{t:schroedereven} Let $(W,S)$ be an even Coxeter system whose nerve $L$ is a flag triangulation of $\BS^3$.  Then $\cH_{i}(\gS(W,S))=0$ for $i\neq 2$.
\end{Thm}

The purpose of this paper is to generalize the methods used in \cite{schroedereven} to any dimension.  Following that template, we look at specific subspaces $\gO$ of $\gS$ called \emph{ruins} (see \ref{ss:ruins}).  What follows is a similar, but much more complicated, statement to that proven by Davis and Okun in \cite{do2}.  For $n=2k$ we consider even Coxeter systems with flag nerves.  We prove that Conjecture \ref{conj:singerc} in dimensions $n-2$ and $n-1$, along with the vanishing of the $\ltwo$-homology of certain subspaces called ``two-letter'' ruins above dimension $k+1$, implies Conjecture \ref{conj:singerc} in dimension $n$.  

%

\section{The Davis Complex}\label{s:davis}
Let $(W,S)$ be a Coxeter system.  Denote by $\cs$ the poset of spherical subsets of $S$, partially ordered by inclusion; and let $\cs^{(k)}:=\{T\in\cs\mid\Card(T)=k\}$.  Given a subset $V$ of $S$, let $\cs_{<V}:=\{T\in \cs\mid T\subset V\}$.  Similar definitions exist for $>, \leq,\geq$.  For any $w\in W$ and $T\in \cs$, we call the coset $wW_{T}$ a \emph{spherical coset}.  The poset of all spherical cosets we will denote by $W\cs$.  

The poset $\cs_{>\emptyset}$ is an abstract simplicial complex, denote it by $L$, and call it the \emph{nerve} of $(W,S)$.  The vertex set of $L$ is $S$ and a non-empty subset of vertices $T$ spans a simplex of $L$ if and only if $T$ is spherical.  

Let $K=|\cs|$, the geometric realization of the poset $\cs$.  In $K$, simplices correspond to linearly ordered chains in the poset $\cs$.  It is the cone on the barycentric subdivision of $L$, the cone point corresponding to the empty set, and thus a finite simplicial complex.  Denote by $\gS(W,S)$, or simply $\gS$ when the system is clear, the geometric realization of the poset $W\cs$.  This is the Davis complex.  The natural action of $W$ on $W\cs$ induces a simplicial action of $W$ on $\gS$ which is proper and cocompact.  $K$ includes naturally into $\gS$ via the map induced by $T \rightarrow W_{T}$, $T\in\cs$.  So we view $K$ as a subcomplex of $\gS$ and note that it is a strict fundamental domain for the action of $W$ on $\gS$.  

For any element $w\in W$, write $wK$ for the $w$-translate of $K$ in $\gS$.  Let $w,w'\in W$ and consider $wK\cap w'K$.  This intersection is non-empty if and only if $V=S(w^{-1}w')$ is a spherical subset.  In fact, $wK\cap w'K$ is simplicially isomorphic to $|\cs_{\geq V}|$, the geometric realization of $\cs_{\geq V}:=\{V'\in\cs\mid V\subseteq V'\}$.

\paragraph{A cubical structure on $\gS$.}  For each $w\in W$, $T\in\cs$, denote by $w\cs_{\leq T}$ the subposet $\{wW_V\mid V\subseteq T\}$ of $W\cs$.  Put $n=\Card(T)$.  $|w\cs_{\leq T}|$ has the combinatorial structure of a subdivision of an $n$-cube.  We identify the sub-simplicial complex $|w\cs_{\leq T}|$ of $\gS$ with this coarser cubical structure and call it a \emph{cube of type $T$}.  Note that the vertices of these cubes correspond to spherical subsets $V\in\cs_{\leq T}$.   (For details on this cubical structure, see \cite{moussong}.)

\paragraph{A cellulation of $\gS$ by Coxeter cells.}  $\gS$ has a coarser cell structure: its cellulation by ``Coxeter cells.''  (References for this cellulation include \cite{davisbook} and \cite{do2}.)  The features of the Coxeter cellulation are summarized by the following from \cite{davisbook}.  

\begin{Prop}\label{p:coxeter} There is a natural cell structure on $\gS$ so that 
\begin{itemize}
\item its vertex set is $W$, its 1-skeleton is the Cayley graph of $(W,S)$ and its 2-skeleton is a Cayley 2-complex.
\item each cell is a Coxeter cell.
\item the link of each vertex is isomorphic to $L$ (the nerve of $(W,S)$) and so if $L$ is a triangulation of $\BS^{n-1}$, $\gS$ is a topological $n$-manifold.
\item a subset of $W$ is the vertex set of a cell if and only if it is a spherical coset and
\item the poset of cells is $W\cs$.
\end{itemize}
\end{Prop}
We will write $\gS_{cc}$, when necessary, to denote the Davis complex equipped with this cellulation by Coxeter cells.  Under this cellulation, the vertices of $\gS_{cc}$ correspond to cosets of $W_{\emptyset}$, i.e. to elements from $W$; and $1$-cells correspond to cosets of $W_{s}$, $s\in S$.  It will be our convention to use the term ``vertices'' for vertices in the cellulation of $\gS$ by Coxeter cells or for vertices in $L$ and to use ``$0$-simplices'' for $0$-simplices in $K$ or translates of $K$.
\vskip.2cm

\subsection{Ruins}\label{ss:ruins}
The following subspaces are defined in \cite{ddjo}.  Let $(W,S)$ be a Coxeter system.  For any $U\subseteq S$, let $\cs(U)=\{T\in\cs|T\subseteq U\}$ and let $\gS(U)$ be the subcomplex of $\gS_{cc}$ consisting of all cells of type $T$, with $T\in\cs(U)$. 

Given $T\in \cs(U)$, define three subcomplexes of $\gS(U)$:
\begin{align}
\gO (U,T):\quad & \text{the union of closed cells of type $T'$, with $T'\in 
\cs(U)_{\geq T}$,}\notag\\
\wh{\gO}(U,T):\quad & \text{the union of closed cells of type $T''$}, 
T''\in \cs(U), T''\notin \cs (U)_{\geq T},\notag\\
\partial\gO (U,T):\quad & \text{the cells of $\gO (U,T)$ of type $T''$,
with } T''\notin \cs (U)_{\geq T}.\notag
\end{align}
The pair $(\gO (U,T),\partial\gO (U,T))$ is called the  \emph{$(U,T)$-ruin}.
For $T=\emptyset$, we have $\gO (U,\emptyset)=\gS (U)$ and
$\partial\gO (U,\emptyset)=\emptyset$.

\paragraph{One-Letter Ruins.}\label{sss:1letter}  Let $t\in S$.  We call the $(S,t)$-ruin a \emph{one-letter ruin}.  Put $U:=\{s\in S\mid m_{st}<\infty\}$, i.e. $U$ is the vertex set of the star of $t$ in $L$.  $1$-cells in $\gO(S,t)$ are of type $u$ where $u\in U$.  So two vertices $w,v$ in a component of $\gO(S,t)$, thought of as group elements of $W$, have the property that $v=wp$, where $p\in W_{U}$.  Thus, the path components of $\gO(S,t)$ are indexed by the cosets $W/W_{U}$.  Denote by $\gO$ the path-component of $\gO(S,t)$ with vertex set corresponding $W_{U}$.  The action of $W_{U}$ on $\gS$ restricts to an action on $\gO$.  Put $K(U):=K\cap\gO$ and note that the $W_{U}$-translates of $K(U)$ cover $\gO$, i.e. $\gO=\bigcup_{w\in W_{U}}wK(U)$.  Let $\partial\gO:=\gO\cap \partial\gO(S,t)$.  Coxeter $1$-cells of $\partial\gO(S,t)$ are of type $u$ where $u\in U-t$; so the path components of $\partial\gO$ are indexed by the cosets $W_U/W_{U-t}$.  

\paragraph{Boundary collars.}\label{sss:boundarycollars}
If we restrict our attention to cubes of type $T$, where $T\subseteq T'$ for some $T'\in\cs_{\geq t}$, $\gO$ is a cubical complex and $\partial\gO$ is a subcomplex.  Moreover, if $B$ is a component of $\partial\gO$, the space $D:=B\times\left[0,1\right]$ is isomorphic to the union of the $w$-translates of $K(U)$ where $w$ is a vertex of $B$.  We call such subspaces \emph{boundary collars}.  It is clear that the collection of boundary collars covers $\gO$.  We denote by $\partial_{in}(D)$ the end of this product which does not lie in $\partial\gO$; the $0$-simplices of $\partial_{in}(D)$ correspond to elements of $\cs_{\geq t}$.  The boundary collars intersect along subsets of these ``inner'' boundaries.  

\paragraph{Two-Letter Ruins.}\label{sss:2letter} For $U\subseteq S$ and $T\in\cs(U)$ with $\Card(T)=2$, we call the $(U,T)$-ruin a ``two-letter'' ruin.

\section{Variations on Singer's Conjecture}\label{s:singervar}
In \cite[Section 8]{do2}, Davis and Okun present several variations of Singer's Conjecture for Coxeter groups (Conjecture \ref{conj:singerc}) in the case $W$ is a right-angled Coxeter system.  They then prove several implications regarding these statements including their proof that Conjecture \ref{conj:singerc} in dimension $2k-1$ implies \ref{conj:singerc} in dimension $2k$.  We procede similarly, beginning with a restatement of Singer's Conjecture.  The Roman numeral notation is to model that used in \cite{do2} and \cite{ddjo}.

\begin{I(n)}\label{conj:singc} Suppose $(W,S)$ is a Coxeter system such that its nerve, $L$, is a triangulation of $\BS^{n-1}$.  Then 
\[\cH_i(\gS(W,S))=0 \text{ for } i\neq\frac{n}{2}.\]
\end{I(n)}

\paragraph{$\I(1)$ and $\I(2)$ are true.}  
Indeed, top-dimensional cycles on manifolds are constant on each component, so a square-summable cycle on an infinite component is constant $0$.  As a result, Conjecture \ref{conj:singer} in dimension $\leq 2$ follows from Poincar\'e duality.  

\paragraph{$\I(3)$ is true.}
In \cite{do2}, the authors show that $\I(3)$ is true for right-angled Coxeter groups.  In \cite{schroedergeom}, the author geometrizes arbitrary $3$-dimensional Davis manifolds and shows that $\I(3)$ follows, Corollary 4.4, \cite{schroedergeom}.

\subsection{Singer's Conjecture for Ruins}\label{ss:scruins}
What follows are variations of $\I(n)$ for one-letter ruins, as definined in section \ref{ss:ruins}.
\begin{II(n)}\label{conj:1ruin} Let $(W,S)$ be a Coxeter system whose nerve $L$ is a triangulation of $\BS^{n-1}$ and let $t\in S$.  Then $\cH_{i}(\gO(S,t),\partial\gO(S,t))=0$ for $i>\frac{n}{2}$.
\end{II(n)}

\begin{V(n)}\label{conj:1ruingen} Let $(W,S)$ be a Coxeter system whose nerve $L$ is a triangulation of $\BS^{n-1}$.  Let $V\subseteq S$ and $t\in V$.  Then $\cH_{i}(\gO(V,t),\partial\gO(V,t))=0$ for $i>\frac{n}{2}$.
\end{V(n)}

\begin{Proposition}\label{p:one-letterexact} $\II(n)$ implies that $\cH_i(\partial\gO(S,t))=\cH_i(\gO(S,t))$ for $i>\frac{n}{2}$. 
\end{Proposition}
\begin{proof} Consider the long exact sequence of the pair $(\gO,\partial\gO)=(\gO(S,t),\partial\gO(S,t)$:

\[
	\ldots\to \cH_{\ast}(\partial\gO)\to
	\cH_{\ast}(\gO)\to
  \cH_{\ast}(\gO,\partial\gO)\to\ldots  
\]

$\II(n)$ implies the third term vanishes, for $i>\frac{n}{2}$.  The result follows from exactness.
\end{proof}

The same proof applied to the long exact sequence of the pair $(\gO(V,t),\partial\gO(V,t))$	proves the following.
\begin{Prop}\label{p:one-letterV} $\V(n)$ implies that $\cH_i(\partial\gO(V,t))=\cH_i(\gO(V,t))$ for $i>\frac{n}{2}$.
\end{Prop}

\paragraph{Singer's Conjecture for two-letter ruins.}  The following statement about two-letter ruins is needed for our program. 

\begin{TR(n)} Let $(W,S)$ be a Coxeter system with nerve $L$ a triangulation of $\BS^{n-1}$.  Let $V\subseteq S$ and let $T\subseteq V$ be a spherical subset with $\Card(T)=2$.  Then $\cH_i(\gO(V,T),\partial\gO(V,T))=0$ for $i>\frac{n}{2}+1$.
\end{TR(n)}

\subsection{Implications}
\paragraph{Excision Isomorphisms.}  Now let $V\subseteq S$, be arbitrary; $T\subseteq V$ spherical, $\gO:=\gO(V,T)$, $\partial\gO:=\partial\gO(V,T)$.  Recall that $\gS(V)$ is the subcomplex of $\gS_{cc}$ consisting of cells of type $T'$, with $T'\subseteq V$.  We have excision isomorphisms (as in \cite{ddjo}):
\begin{equation}\label{e:excision1}
	C_{\ast}(\gO(V,T),\partial\gO)\cong C_{\ast}(\gS(V),\wh{\gO}(V,T)),
\end{equation}
and for any $s\in T$ and $T':=T-s$, 
\begin{equation}\label{e:excision2}
	C_{\ast}(\gS (V-s),\wh{\gO}(V-s,T'))\cong
	C_{\ast}(\wh{\gO}(V,T),\wh{\gO}(V,T')).
\end{equation}
Set $\wh{\gO}:=\wh{\gO}(V,T)$, and $\wh{\gO}':=\wh{\gO}(V,T')$.  Consider the long, weakly exact sequence of the triple $(\gS(V),\wh{\gO},\wh{\gO}')$:
\[
	\ldots\to \cH_{\ast}(\wh{\gO},\wh{\gO}')\to
	\cH_{\ast}(\gS(V),\wh{\gO}')\to
  \cH_{\ast}(\gS(V),\wh{\gO})\to\ldots  
\]
By equations (\ref{e:excision1}) and (\ref{e:excision2}), the left hand term excises to the homology of the $(V-s,T')$-ruin, the right hand term to that of the $(V,T)$-ruin and the middle term to that of the $(V,T')$-ruin; leaving the sequence:
\begin{equation}\label{e:ruinsequence}
	\ldots\to \cH_{\ast}(\gO(V-s,T'),\partial)\to
	\cH_{\ast}(\gO(V,T'),\partial)\to
	\cH_{\ast}(\gO(V,T),\partial)\to\ldots
\end{equation}

\begin{Proposition}\label{p:one-letter} $\left[\II(n) \text{ and } \TR(n)\right]\Longrightarrow \V(n)$.
\end{Proposition}
\begin{proof}  It is clear that $\cH_{i}(\gO(V,t))=0$ for $i>\frac{n}{2}$ whenever $\Card(V)\leq 2$, so we may assume that $\Card(V)>2$.  We induct on $\Card(S-V)$, $\II(n)$ giving us the base case.  Let $V=V'\cup s$ and $t\in V'$.  Assume the result holds for $V$.  If $m_{st}=\infty$ then $(\gO(V',t),\partial)=(\gO(V,t),\partial)$ and we are done.  Otherwise, consider the sequence in equation (\ref{e:ruinsequence}), taking $T=\{s,t\}$, $T'=\{t\}$:
\[
\begin{array}{ccccccc}
0 & \rightarrow & \cH_{n}(\gO(V',t),\partial) & \rightarrow &  \cH_{n}(\gO(V,t),\partial) & \rightarrow &  \cH_{n}(\gO(V,\{s,t\}),\partial) \\
& \rightarrow & \cH_{n-1}(\gO(V',t),\partial) & \rightarrow & \cH_{n-1}(\gO(V,t),\partial) & \rightarrow & \cH_{n-1}(\gO(V,\{s,t\}),\partial) \\
 & & \vdots & & \vdots & & \vdots \\
& \rightarrow & \cH_{k}(\gO(V',t),\partial) & \rightarrow  & \cH_{k}(\gO(V,t),\partial) & \rightarrow & \ldots 
\end{array}
\]
(where $k=\frac{n}{2}+1$ if $n$ is even, $k=\frac{n+1}{2}$ if $n$ is odd).  $\cH_{i}(\gO(V,t),\partial)=0$ for $i>\frac{n}{2}$ by assumption and $\TR(n)$ implies $\cH_{i}(\gO(V,\{s,t\}),\partial)=0$.  So by exactness,  $\cH_{i}(\gO(V',t),\partial)=0$ for $i>\frac{n}{2}$.
\end{proof}

\begin{Theorem}\label{t:Singer} $\V(n) \Longrightarrow \I(n)$.
\end{Theorem}
\begin{proof} Let $V\subseteq S$ and $t\in V$.  Consider the following form of (\ref{e:ruinsequence}), where $T=\{t\}$:
\[
\begin{array}{cccccccc}
	0 & \to & \cH_{n}(\gS(V-t)) & \to & \cH_{4}(\gS(V)) & \to & \cH_{4}(\gO(V,t),\partial) & \to \\
 & \to & \cH_{3}(\gS(V-t)) & \to & \cH_{3}(\gS(V)) & \to & 	\cH_{3}(\gO(V,t),\partial) & \to \\
 & & \vdots & & \vdots & & \vdots \\
 & \to & \cH_{k}(\gS(V-t)) & \to & \cH_{k}(\gS(V)) & \to & \cH_{k}(\gO(V,t),\partial) & \to 
\end{array}
\]
(where $k=\frac{n}{2}+1$ if $n$ is even, $k=\frac{n+1}{2}$ if $n$ is odd).  By $\V(n)$, $\cH_{i}(\gO(V,t),\partial)=0$ for $i>\frac{n}{2}$.  So by exactness,
\[
\cH_{i}(\gS(V-t))\cong \cH_{i}(\gS(V)),
\]
for $i>\frac{n}{2}$.  It follows that $\cH_{i}(\gS)\cong \cH_{i}(\gS(\emptyset))=0$ for $i>\frac{n}{2}$ and hence, by Poincar\'e duality, $\cH_{i}(\gS)=0$ for $i\neq \frac{n}{2}$.
\end{proof}

\section{Even Coxeter systems}\label{s:coxeter}
\subsection{The Compbinatorics of Even systems}\label{ss:combin}
We present some of the background for the combinatorial arguments used in \cite{schroedereven}.  Let $(W,S)$ be a Coxeter system.  Given a subset $U$ of $S$, define $W_{U}$ to be the subgroup of $W$ generated by the elements of $U$.  $(W_U,U)$ is a Coxeter system.  A subset $T$ of $S$ is \textit{spherical} if $W_T$ is a finite subgroup of $W$.  In this case, we will also say that the subgroup $W_{T}$ is spherical.  We say the Coxeter system $(W,S)$ is \textit{even} if for any $s,t\in S$ with $s\neq t$, $m_{st}$ is either even or infinite.

Given $w\in W$, we call an expression $w=(s_{1}s_{2}\cdots s_{n})$ \emph{reduced} if there does not exist an integer $m<n$ with $w=(s'_{1}s'_{2}\cdots s'_{m})$.  Define the \emph{length of $w$}, $l(w)$, to be the integer $n$ such that $(s_{1}s_{2}\cdots s_{n})$, is a reduced expression for $w$.  Denote by $S(w)$ the set of elements of $S$ which comprise a reduced expression for $w$.  This set is well-defined, \cite[Proposition 4.1.1]{davisbook}.

For $T\subseteq S$ and $w\in W$, the coset $wW_{T}$ contains a unique element of minimal length.  This element is said to be $(\emptyset, T)$-reduced.  Moreover, it is shown in \cite[Ex. 3, pp. 31-32]{bourbaki}, that an element is $(\emptyset, T)$-reduced if and only if $l(wt)>l(w)$ for all $t\in T$.  Likewise, we can define the $(T,\emptyset)$-reduced elements to be those $w$ such that $l(tw)>l(w)$ for all $t\in T$.  So given $X,Y\subseteq S$, we say an element $w\in W$ is $(X,Y)$-reduced if it is both $(X,\emptyset)$-reduced and $(\emptyset,Y)$-reduced.

\paragraph{Shortening elements of $W$.}\label{sss:shortening} We have the so-called ``Exchange'' (\textbf{E}) condition for Coxeter systems (\cite[Ch 4. Section 1, Lemma 3]{bourbaki} or \cite[Theorem 3.3.4]{davisbook}): 
\begin{itemize}
	\item\label{i:exchange} (\textbf{E}) Given a reduced expression $w=(s_1\cdots s_k)$ and an element $s\in S$, either $\ell(sw)=k+1$ or there is an index $i$ such that 
	\[sw=(s_1\cdots\wh{s_i}\cdots s_k).\]
\end{itemize}
In the case of even Coxeter systems, the parity of a given generator in the set expressions for an element of $W$ is well-defined.  (We prove this herein, Lemma \ref{l:T-hom}.)  So, in $(\textbf{E})$, $s_i=s$; i.e, if an element of $s\in S$ shortens a given element of $W$, it does so by deleting an instance of $s$ in an expression for $w$.

It is also a fact about Coxeter groups (\cite[Theorem 3.4.2]{davisbook}) that if two reduced expressions represent the same element, then one can be transformed into the other by replacing alternating subwords of the form $(sts\ldots)$ of length $m_{st}$ by the alternating word $(tst\ldots)$ of length $m_{st}$.  The proof of the first of the following two lemmas follows immediately from this.  The proof of the second depends on the first Lemma \ref{l:onereduction} and may be found in \cite{schroedereven}.

\begin{Lem}\label{l:onereduction} Let $t\in S$, $w\in W_{S-t}$ and $v\in W$ with $wtv$ reduced.  If there exists an $r\in S(w)-S(v)$ with $(rt)^2\neq 1$, then all $r$'s appear to the left of all $t$'s in any reduced expression for $wtv$.
\end{Lem}

\begin{Lemma}\label{l:reduction} Let $(W,S)$ be an even Coxeter system, let $t,s\in S$ be such that $2<m_{st}<\infty$ and let $U_{st}=\{r\in S\mid m_{rt}=m_{rs}=2\}$.  Suppose that $tstw'=wtv$ (both reduced) where $w'\in W$, $w\in W_{S-t}$ and $S(v)\subset U_{st}\cup\{s,t\}$.  Then $S(w)\subseteq U_{st}\cup\{s\}$. 
\end{Lemma}
%

\subsection{Coloring the Davis Complex}
Here and for the remainder of this section, we require that $(W,S)$ be an even Coxeter system with nerve $L$.  Fix $t\in S$ and let $U:=\left\{s\in S\mid m_{st}<\infty\right\}$, and let $\gO$ and $\partial\gO$ be defined as in Section \ref{s:davis}.  The following is a generalization of the argument put forth in \cite{schroedereven}.

Any $s\in U$ has the property that $m_{st}<\infty$.  Let $S':=\{s\in U\mid m_{st}>2\}$, and assume that $S'$ is not empty.  The group $W_{U}$ has the following properties.

\begin{Lem}\label{l:srel} Suppose that $L$ is flag.  Then for $s,s'\in S'$, either $s=s'$, or $m_{ss'}=\infty$.
\end{Lem}
\begin{proof} Suppose that $s\neq s'$ and that $m_{ss'}<\infty$.  Then $\{s,s'\}\in \cs$, and since $s,s'$ are both in $U$, the vertices corresponding to $s$, $s'$ and $t$ are pairwise connected in $L$. $L$ is a flag complex, so this implies that $\{s,s',t\} \in \cs$.  But
\[ \frac{1}{m_{ss'}}+\frac{1}{m_{st}}+\frac{1}{m_{ts'}}\leq \frac{1}{m_{ss'}}+\frac{1}{4}+\frac{1}{4}\leq 1.\]
This contradicts $\{s,s',t\}$ being a spherical subset.  So we must have that $m_{ss'}=\infty$.
\end{proof}

\begin{Cor}\label{c:commute} Let $s\in S'$ and let $T\in \cs_{\geq\{s,t\}}$.  Then $m_{ut}=m_{us}=2$ for $u\in T-\{s,t\}$.  
\end{Cor}
In other words, the generators from $T-\{s,t\}$ commute with both $s$ and $t$.

\paragraph{Links.}  Now let $L_{st}$ denote the link in $L$ of the edge connecting the vertices $s$ and $t$.  The above Corollary states that the generators in the vertex set of $L_{st}$ commute with both $s$ and $t$.  As in Lemma \ref{l:reduction}, denote this set of generators by $U_{st}$.  

Of particular interest to us will be elements of $W_U$ with a reduced expression of the form $tst\cdots st$ for some $s\in S'$.  Since $W$ is even, this expression is unique, and we have the following lemma.  

\begin{Lem}\label{l:XY-reduced} Let $s\in S'$ and let $u\in W_{\{s,t\}}$ be such that $u=tst\cdots st$, is a reduced expression beginning and ending with $t$.  Then $u$ is $(U-t,U-t)$-reduced.
\end{Lem}

\begin{Lem}\label{l:T-hom}  Let $V,T\subseteq S$ and consider the function $g_{VT}:W_{V}\rightarrow W_{T}$ induced by the following rule: $g_{VT}(s)=s$ if $s\in V\cap T$ and $g_{VT}(s)=e$ (the identity element of $W$) for $s\in V-T$.  Then $g_{VT}$ is a homomorphism.
\end{Lem}
\begin{proof}  We show that $g_{VT}$ respects the relations in $W_{V}$.  Let $s,u\in V$ be such that $(su)^m=1$.  Then 
\begin{equation*}
	g_{VT}((su)^m)=
	\begin{cases}
		(su)^m & \text{ if } s\in T, u\in T\\
		s^m & \text{ if } s\in T, u\notin T\\
		u^m & \text{ if } u\in T, s\notin T\\
		e & \text{ if } s\notin T, u\notin T.
	\end{cases}
\end{equation*}
In all cases, since $(W_V,V)$ is even, $g_{VT}((su)^m)=e$.  
\end{proof}

\paragraph{Group Action on Cosets.}  Then with $T\in \cs_{\geq t} $ and $U$ as above, we define an action of $W_U$ on the set of cosets $W_T/W_{T-t}$: For $w\in W_{U}$ and $v\in W_{T}$, define 
\begin{equation}\label{e:action}
	w\cdot vW_{T-t}=g_{UT}(w)vW_{T-t}.
\end{equation}

\paragraph{Painting vertices of $\gO$.}  Set
\begin{equation*}
	A=\prod_{T\in\cs_{\geq t}}W_{T}/W_{T-t}.
\end{equation*}
We call $A$ the set of colors and note that it is a finite set.  The action defined in equation (\ref{e:action}) extends to a diagonal $W_{U}$-action on $A$; for $w\in W_{U}$ and $a\in A$, write $w\cdot a$ to denote $w$ acting on $a$.  Let $\bar{e}$ be the element of $A$ defined by taking the trivial coset $W_{T-t}$ for each $T\in \cs_{\geq t}$.  Vertices of $\gO$ correspond to group elements of $W_{U}$, so we paint the vertices of $\gO$ by defining a map $c:W_{U}\rightarrow A$ with the rule $c(w):=w\cdot\bar{e}$.

\begin{Rem}\label{r:trivial} If an element $w\in W_{U}$ does not contain $t$ in any reduced expression, then $w$ acts trivially on the element $\bar{e}$, i.e. $w\cdot\bar{e}=\bar{e}$.
\end{Rem}  

\paragraph{Painting boundary collars.}  We paint the space $wK(U)$ with $c(w)$.  In this way, all of $\gO$ is colored with some element of $A$.  For vertices $w$ and $w'$ of the same component $B$ of $\partial\gO$, $h=w^{-1}w'\in W_{U-t}$, so $c(w')=c(wh)=wh\cdot\bar{e}=w\cdot\bar{e}=c(w)$, where the third equality follows from Remark \ref{r:trivial}.  Therefore all of the boundary collar containing $w$ is painted with $c(w)$.  Note that each component of $\partial\gO$ is monochromatic while the interior of $\gO$ is not.

\begin{Lem}\label{l:samecolordisjoint} Let $D=B\times\left[0,1\right]$ and $D'=B'\times\left[0,1\right]$ be boundary collars where $B$ and $B'$ are different components of $\partial\gO$. Suppose that the vertices of $B$ and $B'$ have the same color.  Then $D\cap D'=\emptyset$.
\end{Lem}
\begin{proof} Suppose, by way of contradiction, that $D\cap D'\neq\emptyset$, i.e. there exist vertices $w\in B$, $w'\in B'$ such that $c(w)=c(w')$ and $wK(U)\cap w'K(U)\neq\emptyset$.  Let $V=S(v)$, where $v=w^{-1}w'$, and since $w$ and $w'$ are from different components of $\partial\gO$, $t\in V$.  Now $c(w)=c(w')\Rightarrow w\cdot\bar{e}=wv\cdot\bar{e}\Rightarrow \bar{e}=v\cdot\bar{e}$. Thus, for any $T\in \cs_{\geq t}$, we have that 
\begin{equation}\label{e:h}
	v\cdot W_{T-t}=W_{T-t}.
\end{equation}
But since $v\in W_{V}$, the action of $v$ on $W_{V}/W_{V-t}$ defined in (\ref{e:action}) is left multiplication by $v$.  But by equation (\ref{e:h}), we have that $v\in W_{V-t}$; a contradiction.
\end{proof}

\paragraph{$c$-collars.} Now for $c\in A$, define the \emph{$c$-collar}, $F_{c}$, to be the disjoint union of the boundary collars $D=B\times\left[0,1\right]$ where each component $B$ of $\partial\gO$ has the color $c$.  The collection of $c$-collars, for all colors $c$, is a finite cover of $\gO$.  

\subsection{Even and odd collars}\label{ss:evenodd}
Let $T=\{t\}$ and consider the homomorphism $g_{UT}:W_{U}\rightarrow W_{t}$ defined in Lemma \ref{l:T-hom}.  Under $g_{UT}$, an element $w\in W_{U}$ is sent to the identity in $W_{t}$ if $w$ has an even number of $t$'s present in some  factorization (and therefore, all factorizations) as a product of generators from $U$ and an element $w\in W_U$ is sent to $t\in W_{t}$ if $w$ has an odd number of $t$'s present in factorizations.  Thus, we call a vertex $w$ \emph{even} if $g_{UT}(w)=e$; \emph{odd} if $g_{UT}(w)=t$.  If two vertices $w$ and $w'$ are such that $c(w)=c(w')$, then clearly $g_{UT}(w)=g_{UT}(w')$, so we may also classify the colors as even or odd.  A $c$-collar is even or odd as $c$ is even or odd and we refer to it as an ``even or odd collar.''  

We will be employing a Mayer-Vietoris argument using the collars as individual pieces of the union.  So, of fundamental importance will be how these collars intersect.  By Remark \ref{r:trivial}, we know that in order for the vertices of a Coxeter cell to support two different colors, this cell must be of type $T\in\cs_{\geq t}$.  But, for a cell to support two different \emph{even} vertices, $v$ and $v'$, this cell must be of type $T\in\cs_{\geq\{s,t\}}$ for exactly one $s\in S'$ (uniqueness is given by Corollary \ref{c:commute}).  Moreover, $w=v^{-1}v'$ has the properties that (1) $\{s,t\}\subseteq S(w)$ and that (2) it contains at least two, and an even number of $t$'s in any factorization as a product of generators.  Such a $w$ we call \emph{$t$-even}.  

\paragraph{The intersection of even collars.}  Now let $L$ be a flag triangulation of $\BS^{n-1}$, so that $\gS$ is an $n$-manifold.  Let $D_0$ denote the boundary collar containing the vertex $e$.  Fix $s\in S'$ and let $D_2$ denote the boundary collar containing the vertex $u$, where $u\in W_{\{s,t\}}$ is $t$-even and has a reduced expression ending in $t$.  We study $D_0\cap D_2$.  

\begin{Lem}\label{l:W'-orbit} Let $W':=W_{U_{st}}$, where $U_{st}=\{r\in S\mid m_{rt}=m_{rs}=2\}$, and let $K'=K(U)\cap uK(U)$.  Denote by $W'K'$ the orbit of $K'$ under $W'$.  Then $D_{0}\cap D_{2} = W'K'$.
\end{Lem}
\begin{proof} For any $w\in W'$, the vertex $w$ is in the same component of $\partial\gO$ as $e$ (by Remark \ref{r:trivial}), and therefore $wK(U)\subset D_0$.  $wu=uw$, so $wu$ is in the same component of $\partial\gO$ as $u$ and $wuK(U)\subset D_{2}$.  Thus $wK'=wK(U)\cap wuK(U)\subset D_{0}\cap D_{2}$.

Now let $\gs$ be a $0$-simplex in $D_{0}\cap D_{2}$.  Then there exist $w,w'\in W_{U-t}$ such that $\gs\in wK(U)\cap uw'K(U)$, i.e. $\gs$ is simultaneously the $w$- and $uw'$-translate of a $0$-simplex $\gs'$ in $K(U)$.  Let $V$ be the spherical subset to which $\gs'$ corresponds and let $v\in W_V$ be such that $uw'=wv$.  $c(e)=c(w)$ and $c(u)=c(uw')$, so $w$ and $uw'$ are differently colored even vertices of a Coxeter cell of type $V$.  By the second paragraph of \ref{ss:evenodd}, $\{s',t\}\subseteq S(v)\subseteq V$ for exactly one $s'\in S'$ and $v$ is $t$-even.  

\textbf{Claim 1}: $s'=s$.\\  
\textbf{Pf}: Since $w'\in W_{U-t}$, $c(u)=c(uw')=c(wv)$, i.e. $u$ and $wv$ act the same on every coordinate of $\bar{e}$.  Consider the $\{s,t\}$-coordinate.  $u\in W_{\{s,t\}}$ is $t$-even, so $u\cdot W_{s}=uW_{s}$ and $uW_{s}\neq W_{s}$.  But if $s\notin S(v)$, then $v$ being $t$-even and $w\in W_{U-t}$ imply that $wv\cdot W_{s}=W_{s}$; which contradicts $u$ and $wv$ having the same color.  So Claim 1 is true, and as a result $V\in\cs_{\geq\{s,t\}}$ and $\gs'\in K'$.  Moreover, by Corollary \ref{c:commute}, $V\subseteq U_{st}\cup \{s,t\}$.  It remains to show that $\gs$ is in the $W'$-orbit of $K'$.  

\textbf{Claim 2}: $S(w)\subseteq (U_{st}\cup \{s\})$.\\
\textbf{Pf}: Take a reduced expression for $u$ which ends in $t$.  If this expression begins with $s$, multiply $u$ on the left by $s$, so that we have $suw'=swv$.  The only change this can effect on $S(w)$ is either adding or subtracting an $s$, which is inconsequential to our claim.  So, we may assume that $u$ has a reduced expression of the form $tst\cdots st$ as described in Lemma \ref{l:XY-reduced}.  Hence, $u$ is $(U-t,U-t)$-reduced and $uw'$ has a reduced expression beginning with the subword $tst$.  $wv$ has a reduced expression of the form $w''tv'$ where $w''\in W_{U-t}$, $S(v')\subset U_{st}\cup\{s,t\}$ and where the difference between $S(w)$ and $S(w'')$ is contained in $U_{st}\cup\{s\}$.  Claim 2 then follows from Lemma \ref{l:reduction} applied to $w''$.  

We now finish the proof of Lemma \ref{l:W'-orbit}.  If $s\notin S(w)$, then $w\in W'$ and we are done since $\gs$ is the $w$-translate of $\gs'$.  If $s\in S(w)$, then $w$ may be written as $qs$, with $q\in W'$ and since $s\in V$, $qsW_{V}=qW_{V}$.  So $\gs$ is also the $q$-translate of $\gs'$.
\end{proof}

\begin{Proposition}\label{p:W'-Daviscpx} $(D_{0}\cap D_{2})\cong\gS(W',U_{st})$, an infinite connected $(n-2)$-manifold.
\end{Proposition}
\begin{proof} Since $S(u)=\{s,t\}$, $K'$ is the geometric realization of the poset 
$\cs_{\geq\{s,t\}}=\{V\in \cs|\{s,t\}\subseteq V\}.$  By Lemma \ref{l:W'-orbit}, $(D_{0}\cap D_{2})\cong |W'\cs_{\geq\{s,t\}}|$, and by Corollary \ref{c:commute}, $\cs_{\geq\{s,t\}}$ is isomorphic to $\cs(U_{st})$ via the map $T\rightarrow T-\{s,t\}$.  So $(D_{0}\cap D_{2})\cong |W'\cs(U_{st})|=\gS(W',U_{st})$. 

Recall that $L_{st}$ denotes the link in $L$ of the edge connecting $s$ and $t$.  Simplices in $L_{st}$ correspond to spherical subsets $T\in\cs$ such that neither $s$ nor $t$ is contained in $T$ but $T\cup\{s,t\}\in \cs$.  So by Corollary \ref{c:commute}, the vertex set of a simplex of $L_{st}$ corresponds to a spherical subset of $\cs(U_{st})$.  Conversely, given a spherical subset $T\in\cs(U_{st})$, $W_{T\cup\{s,t\}}=W_{T}\times W_{\{s,t\}}$, which is finite.  So $T$ corresponds to a simplex of $L_{st}$.  Thus, $L_{st}$ is the nerve of the Coxeter system $(W',U_{st})$.  Since $L$ triangulates $\BS^{n-1}$, $L_{st}$ triangulates $\BS^{n-3}$.  It follows from Proposition \ref{p:coxeter} that $\gS(W',U_{st})$ is a contractible $(n-2)$-manifold.
\end{proof} 

\begin{Cor}\label{c:2evens} Let $c,c'\in A$ be even.  Then $(F_c\cap F_{c'})$ is a disjoint union of infinite $(n-2)$-manifolds.
\end{Cor}
\begin{proof} Suppose that $F_c\neq F_{c'}$ are both even collars and $F_c\cap F_{c'}\neq\emptyset$.  Then there exist even vertices $v$ and $v'$ with $vK(U)\cap v'K(U)\neq\emptyset$.  Let $w=v^{-1}v'$ and put $T=S(v^{-1}v')$.  $T$ is a spherical subset, and $v$ and $v'$ are both vertices of a cell of type $T$.  So we have exactly one $s\in S'$ with $\{s,t\}\subseteq T$.  Factor $w$ as $w=xq$ where $x\in W_{\{s,t\}}$ is $t$-even and $q\in W_{T-\{s,t\}}$.  Now, $x$ may not have a reduced expression ending in $t$.  If it does not, then $xs$ does and it is in the same boundary collar as $x$ and $w$.  So let 
\[
u=
\begin{cases}
	x & \text{ if $x$ has a reduced expression ending in $t$},\\
	xs & \text{ otherwise}.
\end{cases}
\]
Then $vK(U)\cap v'K(U)\subseteq vK(U)\cap vuK(U)$.  Act on the left by $v^{-1}$ and we are in the situation studied in Lemma \ref{l:W'-orbit} and Proposition \ref{p:W'-Daviscpx}.  So $F_c\cap F_{c'}$ is the disjoint union of infinite connected $2$-manifolds.
\end{proof}

\begin{Rem}\label{r:oneevencolor} If $W$ is right-angled, or if $S'=\emptyset$, then $W_U=W_{U-t}\times W_t$ and there is one even and one odd collar.  This is why all the effort on the colors....because the ruins have branching points.
\end{Rem}

\paragraph{Multiple even collars.}  Suppose that $D_{1}, D_{2},\ldots,D_{n}, D_{e}$ are even boundary collars.  Then 
\begin{equation*}
D_{e}\cap \left(\bigcup^{n}_{j=1} D_{j}\right)= (D_{e}\cap D_{1})\cup\cdots\cup (D_{e}\cap D_{n}),
\end{equation*}
and suppose that for some $1\leq i< k\leq n$ we have that $(D_{e}\cap D_{i})$ and $(D_{e}\cap D_{k})$ are not disjoint.  Let $\gs$ be a $0$-simplex contained in $D_{e}\cap D_{i}\cap D_{k}$ corresponding to a coset of the form $vW_{T}$.  Then there exists $w,w'\in W_{T}$ such that $v\in D_{e}$, $vw\in D_{i}$, $vw'\in D_{k}$ and $\gs\in vK(U)\cap vwK(U)\cap vw'K(U)$.  These three vertices are differently colored even vertices of a cell of type $T$, so $\{s,t\}\subseteq T$ for exactly one $s\in S'$ and both $w$ and $w'$ are $t$-even.  Then, as in the proof of Corollary \ref{c:2evens}, it follows that $D_{e}\cap D_{i}=D_{e}\cap D_{k}\cong |W'\cs_{\geq\{s,t\}}|$.  As a result, Corollary \ref{c:2evens} generalizes to the following: 

\begin{Cor}\label{c:multevens} Let $F_{c_1}$, $F_{c_2},\ldots, F_{c_n}, F_{c_e}$ be even collars.  Then 
\begin{equation*}
 \left(F_{c_e}\cap \left(\bigcup^{n}_{j=1} F_{c_j}\right)\right)
\end{equation*}
is a disjoint collection of infinite $(n-2)$-manifolds.
\end{Cor}

\paragraph{Odd collars.}  We now consider how the odd collars intersect with the entire collection of even collars.

\begin{Lem}\label{l:oneodd} Define   
\[\partial_{in}(F_{c}):=\coprod_{D\subset F_{c}}\partial_{in}(D).
\]
Let $\cf_{E}$ denote the union of all even collars and let $F_o$ be an odd collar, then $F_{o}\cap \cf_{E}=\partial_{in}(F_{o})$. 
\end{Lem}
\begin{proof} Since $F_{o}$ is a disjoint union of boundary collars, it suffices to show that $D\cap \cf_{E}=\partial_{in}(D)$ for some boundary collar $D\subset F_{o}$.

($\supseteq$): Let $\gs$ be a $0$-simplex in $\partial_{in}(D)$.  Then $\gs$ corresponds to a coset of the form $wW_{V}$ where $V\in \cs_{\geq t}$ and $w\in W_{U}$ is an odd vertex of $D$.  Consider the even vertex $wt$.  Then since $t\in V$, $wW_{V}=wtW_{V}$, and $\gs\in wtK(U)\subset \cf_{E}$.

($\subseteq$): Now suppose that $\gs$ is a $0$-simplex contained in $D\cap \cf_{E}$.  Then there exists a spherical subset $V$ and cosets $wW_V=w'W_V$ where $w$ is odd and $w'$ is even.  Let $v=w^{-1}w'$.  Since $w$ is odd and $w'$ is even, $v$ must contain an odd number of $t$'s in any of its reduced expressions.  Therefore $t\in V$ and $\gs\in\partial_{in}(D)$.
\end{proof}

As before, let $\cf_{E}$ denote the union of all even collars, and now let $\cf_{O}$ denote the union of a sub-collection of odd collars.  Let $\cf_{E'}=\cf_{E}\cup \cf_{O}$ and let $F_{o}$ be an odd collar not included in $\cf_{O}$.  Then by Lemma \ref{l:oneodd}, 
\begin{equation*}
	F_{o}\cap \cf_{E'}= (F_{o}\cap \cf_{E})\bigcup (F_{o}\cap \cf_{O})=\partial_{in}(F_{o})\bigcup (F_{o}\cap \cf_{O}).
\end{equation*}
Any $0$-simplex in $F_o$ which is also in a different collar must be of the form $wW_V$, where $w$ is a vertex of $F_o$ and $V\in\cs_{\geq t}$.  Therefore $(F_{o}\cap \cf_{O})\subset \partial_{in}(F_{o})$ and $F_{o}\cap \cf_{E'}=\partial_{in}(F_{o})$.  

It is clear from the product structure on boundary collars that $\partial_{in}(F_o)\cong F_o\cap\partial\gO$, the latter a disjoint collection of components of $\partial\gO$.  Since $L$ is flag, we have a 1-1 correspondence between Coxeter cells of any component of $\partial\gO$ and cells of $\gS(W_{U-t},U-t)_{cc}$.  Denote by $L_t$ the link in $L$ of the vertex corresponding to $t$, it is a triangulation of $\BS^{n-2}$ and it is isomorphic to the nerve of $(W_{U-t},U-t)$.  So we have the following corollary.
\begin{Cor}\label{c:oddn-1} Let $\cf_{E'}$ and $F_o$ be as above.  Then $F_o\cap\cf_{E'}$ is a disjoint collection of $(n-1)$-manifolds.
\end{Cor}

\begin{Ex}\label{ex:dim2} The following is representative of our situation.  Suppose $L=\BS^1$, and $U=\{t,r,s\mid (rt)^2=1, (st)^4=1\}$.  $\gO$ is represented in Figure \ref{fig:ruincolors9}.  The black dots represent the vertices of the Coxeter cellulation, with the vertices $e$ and $tst$ labeled.  The even collars are shaded.  Even boundary collars intersect in a $0$-simplex corresponding to the spherical subset $\{s,t\}$.  The intersection of one odd collar and all evens is the inner boundary of the odd collar.  
\end{Ex} 
\begin{figure}[h]
	\centering
		\includegraphics[width=9cm]{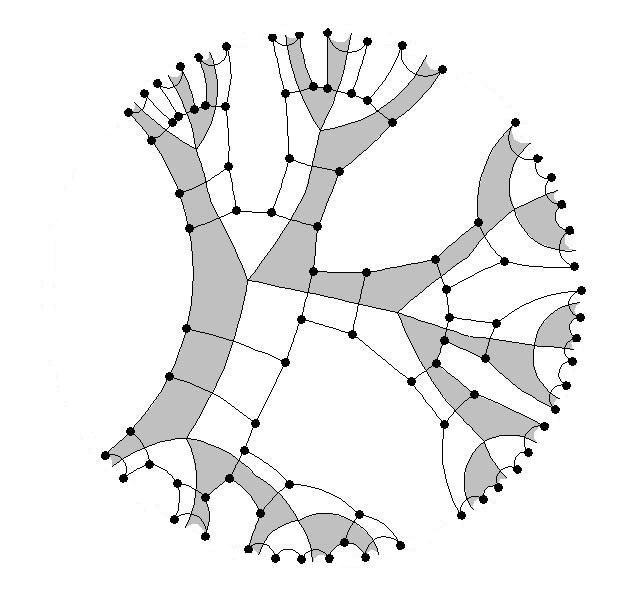}
	\caption{Even and Odd Colors of $\gO$}
	\label{fig:ruincolors9}
\end{figure}

\subsection{Inductive arguments in the case $(W,S)$ is even}\label{ss:evenconj}
Consider the following restatements of conjectures $\I(n)$, $\II(n)$ and $\V(n)$, each in the case that $(W,S)$ is \emph{even} and $L$ is a \emph{flag} triangulation of $\BS^{n-1}$.  (Here the ``E'' stands for even, the ``F'' for flag.)

\begin{EFI(n)}\label{conj:efn} Let $(W,S)$ be an even Coxeter system whose nerve, $L$ is a flag triangulation of $\BS^{n-1}$.  Then 
\[\cH_{i}(\gS)=0 \text{ for } i\neq \frac{n}{2}.\]
\end{EFI(n)}

\begin{EFII(n)}\label{conj:1efruin} Let $(W,S)$ be an even Coxeter system whose nerve $L$ is a flag triangulation of $\BS^{n-1}$ and let $t\in S$.  Then $\cH_{i}(\gO(S,t),\partial\gO(S,t))=0$ for $i>\frac{n}{2}$.
\end{EFII(n)}

\begin{EFV(n)}\label{conj:1efruingen} Let $(W,S)$ be an even Coxeter system whose nerve $L$ is a flag triangulation of $\BS^{n-1}$.  Let $V\subseteq S$ and $t\in V$.  Then $\cH_{i}(\gO(V,t),\partial\gO(V,t))=0$ for $i>\frac{n}{2}$.
\end{EFV(n)}

\begin{EFTR(n)}\label{conj:eftrn}  Let $(W,S)$ be an even Coxeter system with nerve $L$ a flag triangulation of $\BS^{n-1}$.  Let $V\subseteq S$ and let $T\subseteq V$ be a spherical subset with $\Card(T)=2$.  Then $\cH_i(\gO(V,T),\partial\gO(V,T))=0$ for $i>\frac{n}{2}+1$.
\end{EFTR(n)}

A version of $\EFTR(4)$ is proven in \cite{schroedereven}, which only requires showing the top dimensional $\ell^2$-homology vanishes and the proof of which only requires the nerve $L$ being flag.  So, the proof given in \cite{schroedereven} does generalize to the following ``top-dimensional only'' version of $\EFTR(n)$.   Note that the following shows that $\EFTR(3)$ is true, and in fact, we can also drop the even hypothesis.

\begin{Proposition}\label{p:2ruintop} Let $n\geq 3$ and $(W,S)$ be a Coxeter system whose nerve $L$ is flag triangulation of $\BS^{n-1}$.  Then $\cH_n(\gO(V,T),\partial\gO(V,T))=0$.
\end{Proposition}
\begin{proof} If $\cs(V)^{(n)}_{>T}=\emptyset$, then $\gO(V,T)$ does not contain $n$-dimensional cells, and we are done.  So assume that $\cs(V)^{(n)}_{>T}\neq\emptyset$.  The codimension 1 faces of $n$-cells of $\gO(V,T)$ are either faces of one other $n$-cell in $\gO(V,T)$ ($\gS$ is an $n$-manifold), or they are free faces, i.e they are not faces of any other $n$-cell in $\gO(V,T)$.

Suppose that cells of type $T'\in\cs(V)^{(n)}_{>T}$ have a co-dimension one face of type $R$ which is a face of another $n$-cell in $\gO(V,T)$ of type $T''$.  Then any relative $n$-cycle must be constant on adjacent cells of type $T'$ and $T''$, where $T'=R\cup \{r\}$, and $T''=R\cup \{s\}$, $R\in\cs(V)_{>T}^{(n-1)}$ and $r,s\in V$.  Since $L$ is flag and $(n-1)$-dimensional, $m_{rs}=\infty$.  So in this case, there is a sequence of adjacent $n$-cells with vertex sets $W_{T'},W_{T''},sW_{T'},srW_{T''},srsW_{T'},srsrW_{T''},\ldots$.  Hence, this constant must be $0$.  

Now suppose that for a given $n$-cell of $\gO(V,T)$, every co-dimension one face is free.  This cell has faces not contained in $\partial\gO(V,T)$, so relative $n$-cycles cannot be supported on this cell.  
\end{proof}

\begin{Rem}\label{r:eftr}
Note the importance of $\EFTR(n)$ to this program, as you see in Theorem \ref{t:newmain} below.  If a generalized version of this could be proved, then the program would move forward to prove further cases of Conjecture \ref{conj:singerc}.
\end{Rem}

\subsection{Inductive Arguments}
We now generalize the steps used in \cite{schroedereven}, presenting inductive arguments on the painted Davis Complex as a partially successful program to prove $\EFI(n)$.  Since the proofs of Proposition \ref{p:one-letter} and Theorem \ref{t:Singer} do not depend on the even nor odd hypotheses, the same proofs give us the following.
\subsubsection{}\label{sss:evenone-letter} $\left[\EFII(n) \text{ and } \EFTR(n)\right]\Longrightarrow \EFV(n)$.

\subsubsection{}\label{sss:evenSinger} $\EFV(n)\Longrightarrow \EFI(n)$.

\begin{Proposition}\label{p:even1ruin} For $k\in\BZ$, $\left[\EFI(2k-2) \text{ and } \EFI(2k-1)\right]\Longrightarrow \EFII(2k)$.
\end{Proposition}
\begin{proof} It suffices to calculate $\cH_{\ast}(\gO,\partial\gO)$.  We first show that $\cH_{n}(\gO,\partial\gO)=0$.  Consider the long exact sequence of the pair $(\gO,\partial\gO)$:
\[\rightarrow\cH_n(\gO)\rightarrow \cH_{n}(\gO,\partial\gO)\rightarrow \cH_{n-1}(\partial\gO)\rightarrow
\]
$\gO$ is an $n$-dimensional manifold with infinite boundary, so $\cH_n(\gO)=0$ and $\cH_{n-1}(\partial\gO)=0$.  Then by exactness, $\cH_{n}(\gO,\partial\gO)=0$.

Now, let $i>\frac{n}{2}$ and let $\cf_{E'}$ denote the union of a collection of even collars or the union of all evens and a collection of odd collars.  Let $F_c$ be a collar not contained in $\cf_{E'}$ where if $\cf_{E'}$ is not all the even collars, require that $F_c$ be an even collar.  Let $\partial_{E'}=\cf_{E'}\cap\partial\gO$ and let $\partial_{F_c}=F_c\cap\partial\gO$.  Note that $\partial_{E'}\cap\partial_{F_c}=\emptyset$ and consider the relative Mayer-Vietoris sequence of the pair $(\cf_{E'}\cup F_c, \partial_{E'}\cup \partial_{F_c})$:
\[\ldots\rightarrow\cH_{i}(\cf_{E'},\partial_{E'})\oplus\cH_{i}(F_c,\partial_{F_c})\rightarrow\cH_{i}(\cf_{E'}\cup F_c,\partial_{E'}\cup\partial_{F_c})\rightarrow\cH_{i-1}(\cf_{E'}\cap F_c)\rightarrow\ldots\]
Assume that $\cH_{i-1}(\cf_{E'},\partial_{E'})=0$.  Each color retracts onto its boundary, so $\cH_{i}(F_c,\partial_{F_c})=0$.  If $F_c$ is even, then the last term vanishes by Corollary \ref{c:2evens} and $\EFI(n-2)$ and since $i-1>\frac{n-2}{2}$, if $F_c$ is odd, then the last term vanishes by \ref{c:oddn-1}, $\EFI(n-1)$ and since for $n$ even, $i>\frac{n}{2}$ implies $i-1>\frac{n-1}{2}$.  In either case, exactness implies that $\cH_{i}(\cf_{E'}\cup F_c,\partial_{E'}\cup\partial_{F_c})=0$.  It follows from induction that $\cH_{3}(\gO,\partial\gO)=0$.
\end{proof}

\begin{Rem}\label{r:whyeven} Note that if $n$ is odd, then with these hypotheses we are unable to guarantee the vanishing of the $\cH_{(n-1)/2}(\cf_{E'}\cap F_c)$ term for odd colors $F_c$.  However, if we knew the inclusion map of the intersection of the painted boundary collars into the direct sum in the Mayer-Vietoris sequence was injective, we wouldn't need both dimensional statements.  Since this ``even-flag'' argument is pretty technical, I am usure of the most general statement that can be made.
\end{Rem}

It is known that $\I(2)$ and $\I(3)$ are true and therefore the more specific statements $\EFI(2)$ and $\EFII(3)$ are true.  The purpose of \cite{schroedereven} is to prove that $\EFI(4)$ is true.  This is done in a manner exactly like that spelled out above, including the fact that $\EFTR(4)$ is true.  Thus the main result of \cite{schroedereven} is generalized by the following statement.

\begin{Theorem}\label{t:newmain} $\left[\EFI(2k-2), \EFI(2k-1) \text{and} \EFTR(2k)\right]\Longrightarrow \EFI(2k)$.
\end{Theorem}
\begin{proof} By Proposition \ref{p:even1ruin}, the first two hypotheses give us that know that $\EFII(2k)$ is true.  Then, along with $\EFTR(2k)$, this implies that $\EFV(2k)$ is true (see \ref{sss:evenone-letter}).  Finally, by \ref{sss:evenSinger}, we can conclude that $\EFI(2k)$ is true.
\end{proof}

\bibliographystyle{abbrv}

\bibliography{mybib}

\end{document}